\documentclass[twocolumn,final]{autart}    

\usepackage{graphicx}
\usepackage{subcaption}
\usepackage{amsmath}

\usepackage{mathtools}
\usepackage{dsfont}
\usepackage{xcolor}
\usepackage{upgreek}
\usepackage{enumerate}
\usepackage{amsfonts}
\usepackage{amssymb}
\usepackage{textcomp}
\usepackage{tikz}
\usepackage{tkz-berge}
\usetikzlibrary{calc}
\usetikzlibrary{shapes,fit}
\usepackage{verbatim}
\usepackage{setspace}
\usetikzlibrary{shapes,arrows}

\DeclareMathOperator{\R}{\mathbb{R}}
\DeclareMathOperator{\im}{im}
\DeclareMathOperator{\rank}{rank}
\DeclareMathOperator{\spa}{span}
\DeclareMathOperator{\sat}{sat}
\DeclareMathOperator{\ess}{ess}

\DeclareMathOperator{\diag}{diag}
\DeclareMathOperator{\sign}{sign}
\DeclareMathOperator{\ones}{{\ensuremath\mathds{1}}}

\newcommand{\funcRdR}{\ensuremath{{f}}}
\newcommand{\funcRdRd}{\ensuremath{{F}}}
\newcommand{\calG}{\mathcal{G}}
\newcommand{\calI}{\mathcal{I}}
\newcommand{\calV}{\mathcal{V}}
\newcommand{\calE}{\mathcal{E}}
\newcommand{\calF}{\mathcal{F}}

\newcommand{\calR}{\mathcal{R}}

\begin{document}

\begin{frontmatter}

\title{Consensus dynamics with arbitrary sign-preserving nonlinearities\thanksref{footnoteinfo}} 

\thanks[footnoteinfo]{This paper was not presented at any IFAC 
meeting. Corresponding author Jieqiang Wei.}

\author[a1,a2]{J. Wei}\ead{jieqiang@kth.se},     
\author[a1]{A.R.F. Everts}\ead{anneroos@gmail.com},               
\author[a1]{M.K. Camlibel}\ead{M.K.Camlibel@rug.nl},  
\author[a1]{A.J. van der Schaft}\ead{A.J.van.der.Schaft@rug.nl}

\address[a1]{Johann Bernoulli Institute for Mathematics and Comp. Science, Univ. of Groningen, P.O. Box 407, 9700 AK, the Netherlands}  
\address[a2]{ACCESS Linnaeus Centre, School of Electrical Engineering, KTH Royal Institute of Technology, 10044, Stockholm, Sweden}             %

\begin{keyword}                           
Multi-agent systems; Consensus; Nonsmooth analysis; port-Hamiltonian systems.               
\end{keyword}

\begin{abstract}\label{s:Abstract}
This paper studies consensus problems for multi-agent systems defined on directed graphs where the consensus dynamics involves general nonlinear and discontinuous functions. 
Sufficient conditions, only involving basic properties of the nonlinear functions and the topology of the underlying graph, are derived for the agents to converge to consensus.
\end{abstract}

\end{frontmatter}

\section{Introduction}\label{s:Introduction}

Nonlinear agreement protocols have recently attracted the attention of many researchers. They may arise due to the nature of the controller, see e.g. \cite{Jafarian2015,Saber2003}, 
or may describe the physical coupling existing in the network, see e.g. \cite{burger2014duality,Monshizadeh15}. 
In this paper, we consider a general nonlinear consensus protocol.
The topology among the agents is assumed to be a directed graph containing a directed spanning tree, 
which for the linear consensus protocol is known to be a sufficient and necessary condition for reaching state consensus.

Previous work related to this paper can be divided into two categories, depending on whether the dynamical systems are continuous or not. 
For the case of continuous dynamical systems, closely related to this paper are \cite{Papachristodoulou2010} and \cite{lin2007}. 
In \cite{Papachristodoulou2010}, a general first-order consensus protocol with a continuous nonlinear function is considered for the case that there is delay in the communication. 
In \cite{lin2007}, the authors considered a nonlinear consensus protocol with Lipschitz continuous functions, under a switching topology. 
For the case of discontinuous dynamical system, \cite{Cortes2006} is one of the major motivations of this paper. 
Nonlinearities of the form of sign functions were considered in \cite{Cortes2006}, where the notion of Filippov solutions is employed. 
However, in order to guarantee asymptotic consensus of the second network protocol in Section 4 of \cite{Cortes2006}, further conditions turn out to be necessary. 
This is formulated as the main result in Section \ref{ss:node} of this paper. 
In \cite{depersis2013}, the authors considered a similar control protocol as in \cite{Cortes2006} in a hybrid dynamical systems framework with a self-triggered communication policy, which avoids the notion of Filippov solutions. 
In addition, in \cite{depersis2013} practical consensus is considered, that is,  consensus within a predefined margin. 
The results presented in \cite{Cortes2006,depersis2013} are restricted to undirected graphs. 
In \cite{Dimos2010}, the authors considered quantized communication protocols within the framework of hybrid dynamical systems, without using the notion of Filippov solutions.

The contribution of this paper is to provide a uniform framework to analyze the convergence towards consensus of a first-order consensus protocol for a very general class of discontinuous nonlinear functions, under the weakest fixed topology assumption, i.e., a directed graph containing a directed spanning tree. The analysis is conducted with the notion of Filippov solutions, and generalizes and corrects the second network consensus protocol in \cite{Cortes2006}.

The structure of the paper is as follows. In Section \ref{s:Preliminaries}, we introduce some terminology and notation in the context of  graph theory and stability analysis of discontinuous dynamical systems. 
The main results are presented in Theorem~\ref{t:main} and Theorem~\ref{t:combined} in Section~\ref{s:Mainresults}.
The general problem is introduced in Section \ref{ss:Problemformulation}, whereafter in Sections~\ref{ss:node} and \ref{ss:edge} two important subcases are considered, which are then combined in Section \ref{ss:combine}.
In Section \ref{s:PHformulation} we study error dynamics corresponding to the systems considered in Sections~\ref{ss:node}, and provide sufficient conditions for the equivalence between the convergence of the error to zero and the convergence of the original states to consensus. 
Conclusions follow in Section \ref{s:conclusion}.

\section{Preliminaries and notations}\label{s:Preliminaries}
In this section we briefly review some notions from graph 
theory, and give some definitions and notation regarding Filippov solutions.

Let $\mathcal{G}=(\mathcal{V},\mathcal{E},A)$ be a weighted digraph with 
node set $\mathcal{V}=\{v_1,\ldots,v_n\}$, 
edge set $\mathcal{E}\subseteq\mathcal{V}\times\mathcal{V}$,
 and  weighted adjacency matrix $A=[a_{ij}]$ with nonnegative adjacency elements $a_{ij}$. 
An edge of $\mathcal{G}$ is denoted by $e_{ij} := (v_i,v_j)$ and we write $\calI=\{1,2,\ldots,n\}$. 
The adjacency elements $a_{ij}$ are associated with the edges of the graph such that $a_{ij}>0$ if and only if $e_{ji} \in \mathcal{E}$, while $a_{ii}=0$ for all $i \in\calI$. For undirected graphs $A=A^\top$.

The set of neighbors of node $v_i$ is 
denoted by $N_i := \{v_j\in\mathcal{V}:e_{ji}\in\mathcal{E}\}$.
For each node $v_i$, its in-degree and out-degree are defined as
\begin{align*}
  \deg_{\rm in} (v_i) = \sum_{j=1}^n a_{ij},  
\qquad
  \deg_{\rm out}(v_i) = \sum_{j=1}^n a_{ji}.  
\end{align*}
The degree matrix of the digraph $\mathcal{G}$ is a diagonal matrix $\Delta$ 
where $\Delta_{ii}=\deg_{\rm  in}(v_i)$. 
The \emph{graph Laplacian} is 
defined as $L=\Delta-A $ and satisfies $L\ones =0$, where $\ones$ is the $n$-vector containing only ones.
We say that a node $v_i$ is \emph{balanced} if  its  in-degree and out-degree 
are equal. The graph $\mathcal{G}$ is called balanced if all of its nodes are balanced or, 
equivalently, if $\ones^\top L=0$.

A directed path from node $v_i$ to node $v_j$ is a chain of edges from $\calE$ 
such that the first edge starts from $v_i$, the last edge ends at  $v_j$ and 
every edge in between starts where the previous edge ends.
If for every two nodes $v_i$ and $v_j$ there is a directed path from $v_i$ to $v_j$, then the graph $\calG$ is called \emph{strongly connected}.
A subgraph $\calG' = (\calV',\calE',A')$ of $\calG$ is called a \emph{directed 
spanning tree} for $\calG$ if $\calV' =\calV $, $\calE' \subseteq \calE$, and for every node $v_i\in 
\calV'$ there is exactly one node $v_j$ such that $e_{ji}\in \calE'$, except for one 
node, which is called the root of the spanning tree. 
Furthermore, we call a node $v\in \calV$ a \emph{root} of $\calG$ if there is a directed 
spanning tree for $\calG$ with $v$ as a root.
In other words, if $v$ is a root of $\calG$, then there is a directed path from 
$v$ to every other node in the graph.
 A digraph $\calG$ is called {\it weakly
connected} if $\calG^o$ is connected, where  $\calG^o$ is the 
undirected graph obtained from
$\calG$ by ignoring the orientation of the edges. 

The multi-dimensional saturation function $\sat$ and sign function $\sign$ 
are defined as follows. For any $x\in\R^n$,
\begin{align}
 \sat(x\,;u^-,u^+)_i  &= \begin{cases}
u^-_i & \textrm{ if } x_i<u^-_i,\\
x_i & \textrm{ if } x_i\in[u^-_i,u^+_i],\\
u^+_i & \textrm{ if } x_i> u^+_i
\end{cases} 
\,\,\, i\in\calI, \\
\sign(x)_i  & = \begin{cases}
-1 & \textrm{ if } x_i< 0,\\
0 & \textrm{ if } x_i=0,\\
1 & \textrm{ if } x_i> 0,
\end{cases}
\,\,\, i\in\calI,
\end{align}
where $u^-$ and $u^+$ are $n$-vectors containing the lower and upper bounds respectively.

With $\mathbb{R}_-$, $\mathbb{R}_+$ and $\R_{\geqslant 0}$ we denote the sets of 
negative, positive and nonnegative real numbers respectively.
The vectors $e_1,e_2,\ldots,e_n$ denote the canonical basis of $\R^n$. The $i$th row and $j$th column of a matrix $M$ are denoted by $M_{i\cdot}$ and $M_{\cdot j}$ respectively.
For the empty set, we adopt the convention that $\max \emptyset=-\infty$. 

In the rest of this section we give some definitions and notations regarding 
Filippov solutions (see, e.g., \cite{cortes2008}). 
Let $\funcRdRd$ be a map from $\R^n$ to $\R^n$, 
and let $2^{\R^n}$ denote the collection of all subsets of $\R^n$. 
The map $\funcRdRd$ is \emph{essentially bounded} if there is a bound $B$ such that $\|F(x)\|_2< B$ for almost every $x \in \R^n$. 
The map $\funcRdRd$ is \emph{locally essentially bounded} if the restriction of $\funcRdRd$ to every compact subset of $\R^n$ is essentially bounded.
The \emph{Filippov set-valued map} of $\funcRdRd$, denoted
$\calF[\funcRdRd]:\R^n\rightarrow 2^{\R^n}$, is given  as
\begin{equation}
 \calF[\funcRdRd](x)\triangleq
\bigcap_{\delta>0}\bigcap_{\mu(S)=0}\overline{\mathrm{co}}\{\funcRdRd(B(x,
\delta)\backslash S)\},
\end{equation}
where $B(x,\delta)$ is the open ball centered at $x$ with radius $\delta>0$, 
$S\subset\R^n$,
$\mu$ denotes the Lebesgue measure and $\overline{\mathrm{co}}$ denotes the convex closure. The zero measure set $S$ is arbitrarily chosen. Hence, the set $\calF[\funcRdRd](x)$ is independent of the value of $\funcRdRd(x)$.
If $F$ is continuous at $x$, then $ \calF[\funcRdRd](x)=\{F(x)\}$.
A \emph{Filippov solution} of the differential equation $\dot{x}(t)=\funcRdRd(x(t))$ on $[0,t_1]\subset\R$ is
an absolutely continuous function $x:[0,t_1]\rightarrow\R^n$ that 
satisfies the differential inclusion
\begin{equation}\label{e:differential_inclusion}
 \dot{x}(t)\in \calF[\funcRdRd](x(t))
\end{equation}
for almost all $t\in[0,t_1]$.
Let $\funcRdR$ be a map from $\R^n$ to $\R$. We use the same definition of regular function as in \cite{Clarke1990optimization} and recall that convex functions are regular.
If $\funcRdR : \R^n \rightarrow \R$ is locally Lipschitz, then its {\it generalized gradient} 
$\partial 
\funcRdR:\R^n\rightarrow 2^{\R^n}$ is defined by 
\begin{equation}
 \partial \funcRdR(x):=\mathrm{co}\{\lim_{i\rightarrow\infty} \nabla 
\funcRdR(x_i) \mid x_i\rightarrow x, x_i\notin S\cup \Omega_{\funcRdR} \},
\end{equation}
where $\nabla$ denotes the gradient operator, $\Omega_{\funcRdR} \subset\R^n$ denotes the set of points where 
$\funcRdR$ is not differentiable, and $S\subset\R^n$ is an arbitrary set of measure zero. ($\partial \funcRdR(x)$ is independent of the choice of $S$ \cite{Clarke1990optimization}.)
Given a  set-valued map $\calF:\R^n\rightarrow 
2^{\R^n}$, the \emph{set-valued Lie derivative}
$\tilde{\mathcal{L}}_{\calF}\funcRdR:\R^n\rightarrow 2^{\R}$ 
of a locally Lipschitz function $\funcRdR:\R^n\rightarrow \R$  with respect to 
$\calF$ at $x$ is 
defined as
\begin{equation*}
\begin{aligned}
 \tilde{\mathcal{L}}_{\calF}\funcRdR(x) := \{a \mid  \exists \nu\in\calF(x) \textnormal{ s.t. } \zeta^\top\nu=a 
 \ \forall \zeta\in \partial \funcRdR(x)\}.
\end{aligned}
\end{equation*}

A Filippov solution $t\mapsto x(t)$ is \emph{maximal} if it cannot be extended forward in time. Since the Filippov solutions of a discontinuous system \eqref{e:differential_inclusion} are not necessarily unique, we need to specify two types of invariant set. A set $\calR\subset\R^n$ is called \emph{weakly invariant} for \eqref{e:differential_inclusion} if, for each $x_0\in \calR$, at least one maximal solution of \eqref{e:differential_inclusion} with initial condition $x_0$ is contained in $\calR$. Similarly, $\calR\subset \R^n$ is called \emph{strongly invariant} for \eqref{e:differential_inclusion} if, for each $x_0\in \calR$, every maximal solution of \eqref{e:differential_inclusion} with initial condition $x_0$ is contained in $\calR$. For more details, see \cite{cortes2008}.

\section{Main results}\label{s:Mainresults}

\subsection{Problem formulation}\label{ss:Problemformulation}
In this work we consider a network of $n$ agents, who communicate according to a communication 
topology given by a weighted directed graph  
$\calG=(\mathcal{V},\mathcal{E},A)$. 
In this network, agent $i$ receives information from agent $j$ if and only if 
there is an edge from node $v_j$ to node $v_i$ in the graph $\calG$.
We denote the state of agent $i$ at time $t$ as $x_i(t) \in \R$, and consider 
the following dynamics for agent $i$
\begin{equation}\label{e:generalsystem}
\dot x_i = f_i(\sum_{j=1}^n a_{ij}g_{ij}(x_j-x_i)) =: h_i(x), 
\end{equation}
where $f_i$ and $g_{ij}$ are functions, 
$a_{ij}$ are the elements of the adjacency matrix $A$.

Each function $f_i$ describes how agent $i$ handles incoming information,
while $g_{ij}$ are concerned with the flow of information along the 
edges of the graph $\calG$. 
All these functions are nonlinear and may have discontinuities, but we will use the concept of sign-preserving functions.

\begin{defn}  \label{de:sign_preserving}
We say that a function $\varphi:\mathbb{R}\rightarrow\mathbb{R}$ is \emph{sign preserving} if $\varphi(0)=0$ and for each $y \in \R \setminus\{0\}$ we have both $y \varphi(y)>0$ and  $\min y \calF[ \varphi](y)>0$.
\end{defn}

Notice that $y \calF[\varphi] (y) $ attains a minimum since it is closed. Examples of sign-preserving functions are e.g., $\sign$ and $\sat$. 
If a function $\varphi$ has only finitely many discontinuities, e.g. when it is piecewise continuous, then the condition $y \calF[ \varphi](y) >0$ only needs to be checked for its discontinuity points.
The condition $\min y \calF[\varphi] (y) > 0$ for all $y\in \R\setminus\{0\}$,  will be illustrated in Example \ref{ex:reviewerexample}.

Throughout this paper, we assume the following.
\begin{assum} \label{as:signAndPWC}
  The functions $f_i$ and $g_{ij}$ are sign preserving, Lebesgue measurable, and locally essentially bounded.  
\end{assum}
To handle possible discontinuities in the right-hand side of 
\eqref{e:generalsystem}, 
we consider Filippov solutions of the differential inclusion
\begin{equation}\label{e:generalsystem_fili}
\dot{x}(t)\in\calF[h](x(t)).
\end{equation}

The existence of a Filippov solution for each initial condition is guaranteed by the Lebesgue measurability and the local essential boundedness of the functions $f_i$ and $g_{ij}$ in Assumption \ref{as:signAndPWC}. In this paper we assume the completeness of Filippov solutions of \eqref{e:generalsystem_fili} for any initial condition. Notice that when the functions $h_i$ are globally bounded, the completeness of Filippov solution of \eqref{e:generalsystem_fili} is guaranteed by Theorem 1 in Chapter 2 $\S$7 in \cite{filippov1988}.  Moreover, by property 3 of Theorem 1 in \cite{paden1987}, we have 
\begin{equation}\label{e:Cartesian product}
\calF[h](x(t))\subset \bigtimes_{i=1}^n \calF[h_i](x(t)). \footnote{$\bigtimes$ denotes the Cartesian product.}
\end{equation} 

The agents of the network are said to achieve \emph{consensus} if they all converge to the same value, that is, 
$\lim_{t\rightarrow \infty} x(t) = \eta \ones $  
for some \emph{constant} $\eta \in \R$, where $x(t) = [x_1(t),\ldots, x_n(t)]^\top$ is a solution of \eqref{e:generalsystem} with $x(0) = x_0$.
It is well known that if all functions $f_i$ and $g_{ij}$ are the identity function, in which case \eqref{e:generalsystem} boils down to the linear consensus protocol,  then the agents 
will achieve consensus iff the graph $\calG$ contains a 
directed spanning tree \cite{Agaev2005,Ren2005}. 
In this work, we investigate the consensus problem for general functions $f_i$ and $g_{ij}$ satisfying Assumption \ref{as:signAndPWC}.
First, in Section \ref{ss:node}, we consider the special case that 
the functions $g_{ij}$ are equal to the identity function, that is 
$\dot{x}_i=f_i(\sum_{j=1}^n a_{ij}(x_j-x_i))$. 
Thereafter, in Section \ref{ss:edge}, we consider the case where the 
functions $f_i$ are the identity function, that is $\dot x_i = \sum_{j=1}^n 
a_{ij}g_{ij}(x_j-x_i)$.
Finally, in section \ref{ss:combine}, we will combine these results.

The following examples motivate the sign-preserving condition by showing what happens if the functions $f_i$ and $g_{ij}$ do not satisfy this property.

\begin{exmp}\label{ex:saturation}
	Consider the following system defined on the graph given in Fig.~\ref{fig:di-2nodesA}
	\begin{equation}\label{e:2node_digraph}
	\begin{aligned}
	\dot{x}_1 &= f_1(0) \\
	\dot{x}_2 &= f_2(x_1-x_2)，
	\end{aligned}
	\end{equation}
	with $f_i(y)=\sat(y ;0,1)$ for $i = 1,2$. 
	Notice that  $f_i$ satisfies $y f_i(y)=0$ for all $y < 0$, and hence $f_i$ is not sign preserving. In this case the existence of a directed spanning tree is not a sufficient condition for 
	convergence to consensus. Indeed, if the initial condition satisfies 
	$x_2(0)>x_1(0)$, then $x_1(t)= 
	x_1(0)$ and $x_2(t)= x_2(0)$ for all $t\geq0$. Hence, the agents do not reach consensus.
\end{exmp}

\begin{exmp} \label{ex:reviewerexample}
Consider the system \eqref{e:2node_digraph} defined 
on the digraph in Fig.~\ref{fig:di-2nodesA} with $f_i$ given by
\begin{equation}
f_i(y)  = \begin{cases}
y+1 & \textrm{ if } y<-1\\
y & \textrm{ if } y \in [-1,1], \\
y-1 & \textrm{ if } y> 1
\end{cases}
 \;\;i=1,2.
\end{equation}
Then the function $f_i$ satisfies $f_i(0) =0$ and $y f_i(y) >0$ for all $y \neq 0$. However, since $\calF[f_i](1) = [0,1]$ and $\calF[f_i](-1) = [-1,0]$, we have that $\min y \calF[f_i](y) >0$ is not satisfied for $y=\pm 1$. 
Hence, $f_i$ is not sign preserving. Consider the point $x^*=[0,1]^\top$, we have 
\begin{equation*}
\calF[h](x^*)= \overline{\mathrm{co}}\{ [0,-1]^\top,[0,0]^\top \},
\end{equation*}
which contains the point $[0,0]^\top$.
Consequently, $x^*$ is an equilibrium point of the differential inclusion $\dot{x}(t)\in\calF[h](x(t))$. For example, the trajectory
\begin{align*}
x_1(t) = 0, \; x_2(t) = 1 + e^{-t}
\end{align*}
is a solution of \eqref{e:2node_digraph} which converges to $x^*$. Therefore, the agents do not reach consensus.
\end{exmp}

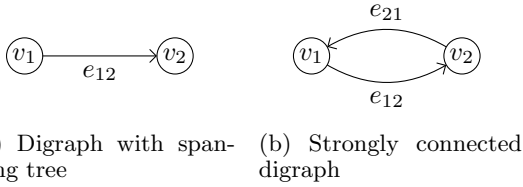
\begin{figure}
        \centering
	\begin{subfigure}[b]{3.5cm}
                \begin{tikzpicture}
		\tikzstyle{EdgeStyle}    = [thin]
		\useasboundingbox (0,0) rectangle (3.5cm,1.5cm);
		\tikzstyle{VertexStyle} = [ minimum size = 20pt,
					                inner sep  = 1pt, 
					                draw = black]

		\Vertex[style={minimum
		size=0.3cm,shape=circle},LabelOut=false,L=\hbox{$v_1$},x=0.7cm,y=0.7cm]{1}
		\Vertex[style={minimum
		size=0.3cm,shape=circle},LabelOut=false,L=\hbox{$v_2$},x=2.7cm,y=0.7cm]{2}
		\draw
		(1) edge[->,>=angle 90,thin,
		]
		node[below]{$e_{12}$} (2);
		\end{tikzpicture}
                \caption{Digraph with spanning tree}
                \label{fig:di-2nodesA}
        \end{subfigure}
        ~ 
        \begin{subfigure}[b]{3.5cm}
                \begin{tikzpicture}
		\tikzstyle{EdgeStyle}    = [thin,
		]
		\useasboundingbox (0,0) rectangle (3.5cm,1.5cm);
		\tikzstyle{VertexStyle} = [draw = black, 					  minimum size = 20pt,%
					  inner sep       = 1pt,]
		\Vertex[style={minimum
		size=0.3cm,shape=circle},LabelOut=false,L=\hbox{$v_1$},x=0.7cm,y=0.7cm]{v1}
		\Vertex[style={minimum
		size=0.3cm,shape=circle},LabelOut=false,L=\hbox{$v_2$},x=2.7cm,y=0.7cm]{v2}
		\draw
		(v1) edge[bend right,->,>=angle 90,thin]
		node[below]{$e_{12}$} (v2)
		(v2) edge[bend right,->,>=angle 90,thin]
		node[above]{$e_{21}$} (v1);
		\end{tikzpicture}
                \caption{Strongly connected digraph}	
                \label{fig:di-2nodesB}
        \end{subfigure}
        ~ 
        \caption{Two digraphs for Examples \ref{ex:saturation}, \ref{ex:twonodesdigraph}, and \ref{ex:tworootsdigraph}.}
	\label{fig:di-2nodes}
\end{figure}

\subsection{Node nonlinearity} \label{ss:node}
We first consider the system \eqref{e:generalsystem} where the functions $g_{ij}$ are all the identity function, 
and focus our attention on the functions $f_i$, which describe how agent $i$ handles the incoming information flow.
In this case, the total dynamics of the agents can be written as 
\begin{equation}\label{e:node_nonlinear}
 \dot{x} = f(-Lx),
\end{equation}
where $L$ is the graph Laplacian induced by the information flow  digraph $\calG = (\calV,\calE,A)$,
and $f(y)=[f_1(y_1),f_2(y_2),\ldots, f_n(y_n)]^\top, \forall y\in\R^n$.
In this case we consider Filippov solutions of the differential inclusion
\begin{equation}\label{e:node_nonlinear_fili}
\dot{x}(t)\in\calF[h](x(t)),
\end{equation}
where $h(x) = f(-Lx)$. 
Note that, since $L$ is a singular matrix, we have $\calF[h](x(t))\neq\calF[f](-Lx(t))$ in general.

The aim of this section is to investigate under which conditions the Filippov 
solutions of the system \eqref{e:node_nonlinear_fili} achieve consensus.
Because of possible discontinuity of the right-hand side 
of \eqref{e:node_nonlinear}, there can be Filippov solutions of 
\eqref{e:node_nonlinear_fili} 
that are unbounded. The following example illustrates this unwanted behavior.

\begin{exmp}\label{example_undi_3nodes}
 Consider a dynamical system \eqref{e:node_nonlinear} 
defined on an undirected graph as given in Fig.~\ref{fig:3nodesA}, 
where the functions $f_i$ are all signum function.
Suppose $x(t_0) \in\spa\{\ones\}$ at time $t_0$, then   
\begin{equation}\label{e:filippovset-example}
\calF[h](x(t_0))=\overline{\mathrm{co}}\left\{ 
\nu_1,\nu_2,\nu_3,-\nu_1,-\nu_2,-\nu_3 \right\},
\end{equation}
where $\nu_1=[1,1,-1]^\top$, $\nu_2=[1,-1,1]^\top$, and $\nu_3=[-1,1,1]^\top$. Since 
$\sum_{i=1}^3 \frac{1}{3} \nu_i =\frac{1}{3}\ones$, we have that 
$\{\eta\ones\mid 
\eta\in[-\frac{1}{3},\frac{1}{3}]\}\subset\calF[h](x(t_0))$. Hence, any function 
$x(t)=\eta(t)\ones$ with $\eta(t)$ differentiable almost everywhere and satisfying $\dot{\eta}(t)\in[-\frac{1}{3},\frac{1}{3}]$ 
is a Filippov solution, e.g., $x(t) = \frac{1}{3}t \mathds{1}$ and $x(t) =  \frac{1}{3} \sin{(t)} \mathds{1}$ which exhibit sliding consensus. 
\end{exmp}

\begin{figure}
        \centering
        \begin{subfigure}[b]{3.5cm}
                \begin{tikzpicture}
		\tikzstyle{EdgeStyle}    = [thin]
		\useasboundingbox (0,0) rectangle (3.5cm,2.2cm);
		\tikzstyle{VertexStyle} = [draw=black,
					  minimum size = 20pt,
					  inner sep       = 1pt,]
		\Vertex[style={minimum
		size=0.3cm,shape=circle},LabelOut=false,L=\hbox{$v_1$},x=1cm,y=0.4cm]{1}
		\Vertex[style={minimum
		size=0.3cm,shape=circle},LabelOut=false,L=\hbox{$v_2$},x=3cm,y=0.4cm]{2}
		\Vertex[style={minimum
		size=0.3cm,shape=circle},LabelOut=false,L=\hbox{$v_3$},x=2cm,y=2.13cm]{3}
		\draw
		(1) edge[-,>=angle 90,thin]
		node[below]{} (2)
		(2) edge[-,>=angle 90,thin]
		node[right]{} (3)
		(3) edge[-,>=angle 90,thin]node[left]{} (1);
		\end{tikzpicture}
                \caption{Undirected graph}
                \label{fig:3nodesA}
        \end{subfigure}
        ~ 
        \begin{subfigure}[b]{3.5cm}
               \begin{tikzpicture}
	      \tikzstyle{EdgeStyle}    = [thin]
	      \useasboundingbox (0,0) rectangle (3.5cm,2.2cm);
	      \tikzstyle{VertexStyle} = [draw = black,
					minimum size = 20pt,
					inner sep       = 1pt,]
	      \Vertex[style={minimum
	      size=0.3cm,shape=circle},LabelOut=false,L=\hbox{$v_1$},x=1cm,y=0.4cm]{1}
	      \Vertex[style={minimum
	      size=0.3cm,shape=circle},LabelOut=false,L=\hbox{$v_2$},x=3cm,y=0.4cm]{2}
	      \Vertex[style={minimum
	      size=0.3cm,shape=circle},LabelOut=false,L=\hbox{$v_3$},x=2cm,y=2.13cm]{3}
	      \draw
	      (1) edge[->,>=angle 90,thin]
	      node[below]{} (2)
	      (2) edge[->,>=angle 90,thin]
	      node[right]{} (3)
	      (3) edge[->,>=angle 90,thin]node[left]{} (1);
	      \end{tikzpicture}
                \caption{Directed graph}
                \label{fig:3nodesB}
        \end{subfigure}
        ~ 
        \caption{Two graphs with three nodes used in Examples 
\ref{example_undi_3nodes}, \ref{ex.condition(i)}, \ref{ex:counterexmaple_edgenonlinear_undirected} 
and \ref{ex:counterexmaple_edgenonlinear_directed}.}
	\label{fig:3nodes}
\end{figure}
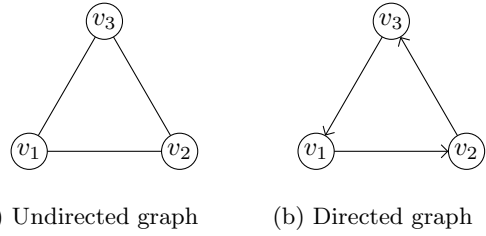

The undesirable behavior $x(t) = \eta(t)\ones$ with $\eta(t)$ a nonconstant function in the previous example will be called \emph{sliding consensus}. Sliding consensus arises whenever  $\eta \ones$ is contained in $\calF [h] (\alpha \ones)$ for some scalars $\eta \neq 0$ and sufficiently many $\alpha$.  Note that this example shows that for the validity of Theorem 11 (ii) in \cite{Cortes2006} we need extra conditions. A counter example to Theorem 11 (i) in \cite{Cortes2006} can be constructed similarly to Example \ref{example_undi_3nodes}.
In fact, it will turn out that the sliding consensus can be excluded by 
replacing the signum function for at least one node by a  function that is continuous at the origin, for example the logarithmic quantizer. This motivates the introduction of the following subsets of the index set $\calI$ corresponding to the digraph $\calG$:
\begin{align*}
\calI_r &=\{i\in\calI \mid v_i \text{ is a root of }\calG\},\\
\calI_c &=\{i\in\calI \mid f_i \text{ is continuous at the origin}\}. 
\end{align*}

Before we present the main result of this section, we first state a preparatory lemma.

\begin{lem}[Prop. 2.2.6, 2.3.6 in \cite{Clarke1990optimization}]\label{regular_Lipschitz_lyapunov}
The following functions are regular and Lipschitz continuous,
\begin{equation}\label{e:VandW}
 V(x):= \max_{i\in\calI} x_i, \qquad W(x):= -\min_{i\in\calI} x_i. 
\end{equation} 
\end{lem}

\begin{thm}\label{t:main}
Consider system \eqref{e:node_nonlinear_fili} defined on a digraph 
$\calG=(\calV,\calE,A)$. If one of the following conditions holds, 
\begin{enumerate}[(i)]
\item $\calI_c\cap\calI_r$ is not empty,
\item  $|\calI_r|=1$,
\item 
$|\calI_r|=2$, $f_i(0^-)$ and $f_i(0^+)$ exist, and $f_i(0^-)=-f_i(0^+)$ for $i\in\calI_r$, 
\end{enumerate}
then all the trajectories of system \eqref{e:node_nonlinear_fili} achieve 
consensus asymptotically, for any initial condition. Furthermore, they will remain in the set $[\min_ix_i(0),\max_i x_i(0)]^n$ for all $t\geq0$.
\end{thm}

\begin{proof}
Notice that in all three cases the index set $\calI_r$ is nonempty, which implies that the graph $\calG$ contains a 
directed spanning tree. Condition~{\it (i)} implies that the digraph 
$\calG$ has a root $v_i$ for which $f_i$ is continuous at the origin.

Consider candidate Lyapunov functions $V$ and $W$ as given in \eqref{e:VandW}.  
Let $x(t)$ be a trajectory of \eqref{e:node_nonlinear_fili} and define 
\[\alpha(t)=\{k\in\calI\mid x_k(t)=V(x(t))\}.\] 
The generalized gradient of $V$ is given as \cite[Example 2.2.8]{Clarke1990optimization}
\begin{equation}
\partial V(x(t))  = \mathrm{co}\{e_k\in\R^n \mid k \in \alpha(t) \}.
\end{equation}

Let $\Psi$ be defined as 
\begin{equation}
\Psi = \{t\geq 0 \mid \textnormal{both } \dot{x}(t) \textnormal{ and } \frac{d}{dt}V(x(t)) \textnormal{ exist} \}.
\end{equation}
Since $x$ is absolutely continuous and $V$ is locally Lipschitz, we have that  $\Psi=\R_{\geq 0}\setminus\bar{\Psi}$ for a set $\bar{\Psi}$ of measure zero. By Lemma 1 in \cite{Bacciotti1999}, we have 
\begin{equation}
\frac{d}{dt}V(x(t))\in \tilde{\mathcal{L}}_{\calF[h]}V(x(t))
\end{equation}
for all $t\in\Psi$ and hence that the set $\tilde{\mathcal{L}}_{\calF[h]}V(x(t))\neq\emptyset$ for all $t\in\Psi$. 
For $t\in\bar{\Psi}$, we have $\tilde{\mathcal{L}}_{\calF[h]}V(x(t))$ is empty, hence $\max \tilde{\mathcal{L}}_{\calF[h]}V(x(t))< 0$.
For $t \in \Psi$, let $a\in\tilde{\mathcal{L}}_{\calF[h]}V(x(t))$. 
By definition, there exists a $\nu^a \in \calF[h](x(t))$ such that $a = \nu^a\cdot \zeta$ for all 
$\zeta\in\partial V(x(t))$. Consequently, by choosing $\zeta = e_k$ for $k \in \alpha(t)$, we observe that $\nu^a$ satisfies 
\begin{equation} \label{e:nu-alpha} \nu_k^a  = a \qquad \forall k \in \alpha(t). \end{equation} 
Next, we want to show that $\max \tilde{\mathcal{L}}_{\calF[h]}V(x(t))\leq 0$ for all $t\in\Psi$ by considering two possible cases: $\calI_r\nsubseteq\alpha(t)$ or 
$\calI_r\subseteq\alpha(t)$. 

If $\calI_r\nsubseteq\alpha(t)$, then there exists an $i\in\calI_r$ such that $x_i(t)<V(x(t))$. 
Furthermore, since $v_i$ is a root, we can choose an index $j\in\alpha(t)$ such that the shortest path from $v_i$ to $v_j$ has the least number of edges. By our choice of $j$, there is at least one edge $e_{kj} \in \calE $ such that $x_k(t)<x_j(t)$, which implies that we have $-L_{j\cdot}x(t)<0$.
Moreover, the existence of an edge $e_{kj}$ implies that $\rank L_{j\cdot}=1$, which together with property 4 of Theorem 1 in \cite{paden1987} gives us
\begin{equation}
\calF[h_i](x(t))=\calF[f_j](-L_{j\cdot}x(t)).
\end{equation}
By the sign-preserving property of $f_j$ and $-L_{j\cdot} x(t) <0$, we have that $\calF[h_j](x(t))\subset\R_-$. By \eqref{e:Cartesian product}, we find that $\nu_j<0$ for any $\nu\in\calF[h](x(t))$.
Using observation \eqref{e:nu-alpha} for $k=j$, we see that every $a\in\tilde{\mathcal{L}}_{\calF[h]}V(x(t))$ satisfies $a<0$. By the fact that  $\tilde{\mathcal{L}}_{\calF[h]}V(x(t))$ is a closed set, we have $\max \tilde{\mathcal{L}}_{\calF[h]}V(x(t))<0$.

If $\calI_r\subseteq\alpha(t)$, we will consider the conditions {\it (i), (ii)} and {\it 
(iii)} separately and prove that $\tilde{\mathcal{L}}_{\calF[h]}V(x(t))=\{0\}$.
First, we note that if a node $v_k$ is a root, then $e_{jk} \in \calE$ implies that $v_j$ is a root as well, and hence we have 
\begin{equation} \label{e:rootneighbors}
(-Lx)_k = \sum_{j\in\calI}a_{kj}(x_j-x_k) = \sum_{j\in\calI_r}a_{kj}(x_j-x_k).
\end{equation}
\begin{enumerate}[(i)]
\item  
In this case $\calI_c\cap\calI_r\subseteq \alpha(t)$. 
For any $i\in\calI_c\cap\calI_r$, we have that $f_i$ is continuous at $0$ and satisfies $f_i(0)=0$. This implies that any $\nu\in\calF[h](x(t))$ satisfies $\nu_i=0$.
Using observation \eqref{e:nu-alpha}, we can conclude that $\tilde{\mathcal{L}}_{\calF[h]}V(x(t))=\{0\}$. 

\item 
Let $\calI_r=\{i\}$.
Since there is only one root in this case, namely $v_i$, we have $L_{i\cdot}=0$ and hence  $f_i((L x(t))_i ) = f_i(0) = 0$ for all $t$. 
Consequently, each $\nu\in\calF[h](x(t))$ satisfies $\nu_i=0$. Using observation \eqref{e:nu-alpha} again, we see that $\tilde{\mathcal{L}}_{\calF[h]}V(x(t))=\{0\}$.

\item  Let $\calI_r=\{i,j\}$. By \eqref{e:rootneighbors}, the dynamics of $x_i$ and $x_j$ are 
given as $\dot{x}_i=f_i(a_{ij}(x_j-x_i))$ and $\dot{x}_j=f_j(a_{ji}(x_i-x_j))$ respectively. 

Since $\calI_r\subseteq\alpha(t)$, we have $x_i(t) = x_j(t)$ and hence any $\nu\in\calF[h](x(t))$ satisfies 
\begin{align*}
[\nu_i, \nu_j]
&\subseteq 
\overline{\mathrm{co}} \{
 [f_i(0^-), f_j(0^+)], 
 [f_i(0^+),  f_j(0^-)]
\} \nonumber  \\
&= 
\overline{\mathrm{co}} \{
[f_i(0^-), f_j(0^+)], -
[f_i(0^-), f_j(0^+)]  \}, 
\end{align*}
where the last equality is implied by condition {\it (iii)}.  This implies that the convex set given in the above equation 
is a line segment that only crosses $\spa\{[1,1]^\top \}$ in the origin. This implies that any $\nu \in \calF[h](x(t))$ with $\nu_i=\nu_j$ must satisfy $\nu_i=\nu_j=0$. Using $\partial V(x(t))  = \mathrm{co}\{e_i, e_j\in\R^n \}$ and \eqref{e:nu-alpha}, we see that $\tilde{\mathcal{L}}_{\calF[h]}V(x(t))=\{0\}$.
\end{enumerate}

Define $\beta(t)=\{i\in\calI\mid x_i(t)=-W(x(t))\}$. By using similar computations and observations, we find that $ 
\max \tilde{\mathcal{L}}_{\calF[h]}W(x(t))<0$ if $\calI_r\nsubseteq\beta(t)$,  
and  $\max \tilde{\mathcal{L}}_{\calF[h]}W(x(t))\leqslant 0$ if 
$\calI_r\subseteq\beta(t)$.

We conclude that $V(x(t))$ and $W(x(t))$ are not increasing along the trajectories $x(t)$ of the 
system (\ref{e:node_nonlinear_fili}). Hence, the trajectories are bounded and remain in the set $[\min_i x_i(0),\max_i x_i(0)]^n$ for all $t\geq0$. 
Therefore, for any $N\in\R_+$, the set $S_N=\{x\in\R^n \mid \|x\|_{\infty}\leqslant N\}$ is 
strongly invariant for (\ref{e:node_nonlinear_fili}). 
By Theorem 2 in \cite{cortes2008}, we 
have that all solutions of (\ref{e:node_nonlinear_fili}) starting at $S_N$ 
converge to 
the largest weakly invariant set $M$ contained in 
\begin{equation}
\begin{aligned}
 S_N & \cap\overline{\{x\in\R^n \mid 
0\in\tilde{\mathcal{L}}_{\calF[h]}V(x)\}} \\
& \cap\overline{\{x\in\R^n \mid 
0\in\tilde{\mathcal{L}}_{\calF[h]}W(x)\}}.
\end{aligned}
\end{equation}
From the argument above we see that $0\in\tilde{\mathcal{L}}_{\calF[h]}V(x(t))$  is only possible if $\calI_r \subseteq \alpha(t)$, and 
$0\in\tilde{\mathcal{L}}_{\calF[h]}W(x(t))$ can only happen if $\calI_r \subseteq \beta(t)$. 
This implies that for every root $v_i$, the state $x_i$ converges simultaneously to the maximum and to the minimum, i.e., the trajectories $x(t)$ of the system achieve consensus for any initial condition.
\end{proof}

\begin{rem}
Here we interpret condition {\it (i)}, {\it (ii)} and {\it (iii)} in Theorem \ref{t:main}. 
It can be seen from the proof of Theorem \ref{t:main} that the sliding consensus can be introduced only by the behaviors of the agents in $\calI_r$, while the agents in $\calI\setminus \calI_r$ just track the trajectories of the root agents. The condition {\it (i)} is a sufficient condition for the general case, namely $\calI_c\cap\calI_r \neq \emptyset$. 
In this case, there is at least one root for which the input converges to zero as all the agents converge to consensus, which prevents sliding behavior.
The conditions {\it (ii)} and {\it (iii)} are provided for two special cases,
 i.e., $|\calI_r|=1$ and $|\calI_r|=2$. In these cases, we allow $\calI_r\cap\calI_c=\emptyset$. When there is only one root in the graph, the states of the other agents will converge to the state of the root which is constant along the trajectory, excluding sliding behavior. 
If there are two roots, $v_i$ and $v_j$, and both $f_i$ and $f_j$ are not continuous at the origin, then condition {\it (iii)} gives another way to eliminate sliding consensus: the limits of $f_i$ and $f_j$ for $t \rightarrow 0^+$ should be opposite to the limits for $t \rightarrow 0^-$. 
For example, $f_i=f_j=\sign$ is one of this type of protocols. 
\end{rem}

\begin{rem}
The set  $\calI_c$ 
 can be enlarged such that it contains all the functions which are \emph{essentially continuous} at the origin, i.e., $\ess \lim_{x_i\rightarrow 0^-} f_i(x_i)=\ess \lim_{x_i\rightarrow 0^+} f_i(x_i)=0$ (for definitions see e.g. \cite{Arutyunov2000,chung2006markov}). 
This can be done since in the definitions of both essential limits and Filippov set-valued map, any zero measure set can be excluded.
For condition {\it (iii)} in Theorem \ref{t:main}, the same extension is possible; considering essential limits in stead of limits.
\end{rem}

The conditions {\it (i)}, {\it (ii)} and {\it (iii)} in Theorem \ref{t:main} 
all exclude the possibility of sliding consensus, and guarantee asymptotic consensus. The role of each condition will be illustrated in the following examples.

\begin{exmp}\label{ex.condition(i)}
Consider system \eqref{e:node_nonlinear} defined on the undirected graph in Fig.~\ref{fig:3nodesA} with all edge weights one.
Suppose that  $f_1$ is continuous at the origin, so that condition {\it (i)} in Theorem \ref{t:main} is satisfied. 
Then the sliding consensus is not a Filippov solution. Indeed, if at time $t_0$ we have $x(t_0)\in\spa\{\ones\}$, then the first 
component of the Filippov set-valued map $\calF[h](x(t_0))$ is 
equal to $\{0\}$. This implies that $x(t) = x(t_0)$, for all $t \geq t_0$.
\end{exmp}

\begin{exmp}\label{ex:twonodesdigraph}
Consider system \eqref{e:2node_digraph} defined 
on the digraph in Fig.~\ref{fig:di-2nodesA}.
It satisfies condition 
{\it (ii)} of Theorem \ref{t:main}. Since $f_1(0)=0$, the state of the root $v_1$ is constant. Consensus is 
achieved by the fact that $f_2$ is sign preserving.
\end{exmp}

\begin{exmp}\label{ex:tworootsdigraph}
 Consider system \eqref{e:node_nonlinear} defined on the digraph given in 
Fig.~\ref{fig:di-2nodesB} with $a_{12} = a_{21} = 1$.

First, we consider a case in which $f_1$ and $f_2$ satisfy condition {\it (iii)} of Theorem \ref{t:main} and take $f_1 = f_2 = \sign(\cdot)$.
If the trajectory achieves consensus at time $t$, the  Filippov set-valued map $ \calF[h](x(t))$ equals 
$\overline{\mathrm{co}}\{[1,-1]^\top,[-1,1]^\top\}$, which intersects 
$\spa\{\ones\}$ only at $[0,0]^\top$. Hence $  \tilde{\mathcal{L}}_{\calF[h]}V(x)  = \tilde{\mathcal{L}}_{\calF[h]}W(x)= 0$, which implies that the trajectory remains 
constant, i.e., there is no sliding consensus.

Second, we consider a case in which $f_i(0^-)\neq -f_i(0^+)$ for $i =1,2$, which means  that the condition {\it (iii)} of Theorem \ref{t:main} is not satisfied.
In this case, sliding consensus can be a Filippov solution. For instance, take
\begin{equation*}
 f_i(x)  = \left\{ \begin{array}{ll}
2 & \textrm{ if } x>0\\
0 & \textrm{ if } x=0,\\
-1 & \textrm{ if } x<0
\end{array}
\,\quad i=1,2. \right.
\end{equation*}
Suppose that at $t_0$ the state $x$ achieves consensus. Then the Filippov set-valued map at 
$x(t_0)$ is $\overline{\mathrm{co}}\{[-1,2]^\top,[2,-1]^\top\}$ which intersects 
$\spa\{\ones\}$ at $[\frac{1}{2},\frac{1}{2}]^\top$. Then 
$x(t)=\frac{1}{2}\ones t + x(t_0)$ is a Filippov solution for $t\geq t_0$ that exhibits sliding consensus.
\end{exmp}

\subsection{Edge nonlinearity}\label{ss:edge}
In this section we consider the case where the functions $f_i$ are all the identity function, that is,
\begin{align}\label{e:edge_nonlinear}
 \dot{x}_i = \sum_{j=1}^N a_{ij}g_{ij}(x_j-x_i) 
 =: h_i(x), \quad i\in\calI. 
\end{align}
We consider two cases, corresponding to the underlying graph $\calG=\{\calV, \calE \}$ being undirected or
directed, starting with the undirected case.
We introduce the following assumption on the functions $g_{ij}$. 

\begin{assum}\label{as:symmetry_origin}
 For all $e_{ji}\in\calE$, the right and left limits of $g_{ij}$ and $g_{ji}$ at the origin exist, and satisfy $g_{ij}(0^-)=-g_{ji}(0^+)$. 
\end{assum}
To illustrate the need of Assumption \ref{as:symmetry_origin}, we give the following example.
\begin{exmp}\label{ex:counterexmaple_edgenonlinear_undirected}
If $g_{ij}(0^-)\neq-g_{ji}(0^+)$, then sliding consensus may occur. For instance, consider the system \eqref{e:edge_nonlinear} defined on the undirected graph in Fig.~\ref{fig:3nodesA} given by 
\begin{align*}
\dot{x}_1(t) &= g_{12}(x_2(t)-x_1(t))+g_{13}(x_3(t)-x_1(t)) \\ 
\dot{x}_2(t) &= g_{21}(x_1(t)-x_2(t))+g_{23}(x_3(t)-x_2(t)) \\
\dot{x}_3(t) &= g_{31}(x_1(t)-x_3(t))+g_{32}(x_2(t)-x_3(t)) 
\end{align*}
where
\begin{equation*}
 g_{ij}(x)  = \begin{cases}
1.5 & \textrm{ if } x>0,\\
0 & \textrm{ if } x=0,\\
-0.5 & \textrm{ if } x<0,
\end{cases}
\quad \forall e_{ji}\in\calE,
\end{equation*}
Suppose that at time $t_0$ the state $x(t_0)\in\spa\{\ones\}$, then $\calF[h](x(t_0))$ is the closed convex hull spanned by $[-1,1,3]^\top$, $[-1,3,1]^\top$, $[1,-1,3]^\top$, $[3,-1,1]^\top$, $[1,3,-1]^\top$ and $[3,1,-1]^\top$. Hence,  $\ones\in\calF[h](x(t_0))$ and thus $x(t)=t\ones+x(t_0)$ is a  Filippov solution for $t>t_0$.
\end{exmp}

Next, we present the main result of this section.

\begin{thm}\label{t:edge_nonlinear}
 Consider the dynamics \eqref{e:edge_nonlinear} defined on a connected 
undirected graph. Suppose the 
functions $g_{ij}$ satisfy Assumptions \ref{as:signAndPWC} and \ref{as:symmetry_origin}. Then the 
trajectories of the system \eqref{e:edge_nonlinear} achieve consensus 
asymptotically.
\end{thm}

\begin{proof}
 Consider the Lyapunov candidate functions $V$ and $W$ as defined in \eqref{e:VandW}. We use the same notations as in the proof of Theorem \ref{t:main}. Similarly, as the proof of Theorem \ref{t:main}, we only prove that $\max\tilde{\mathcal{L}}_{\calF[h]}V(x(t))\leq0$ for all $t\in\Psi$ where $\mu(\R_{\geq 0}\setminus\Psi)=0$ and the set $\tilde{\mathcal{L}}_{\calF[h]}V(x(t))\neq\emptyset$ for all $t\in\Psi$. 
 
 By introducing the functions $\varphi_{ji}(x)=(e_j-e_i)^\top x$ for $i,j\in\calI$, the function $h_i(x)$ in \eqref{e:edge_nonlinear} can be rewritten as
 \begin{equation}
 h_i(x)=\sum_{j=1}^{n}a_{ij}(g_{ij}\circ \varphi_{ji})(x).
 \end{equation} 
 Then, using Theorem 1 in \cite{paden1987}, we see that the Filippov set-valued map $\calF[h](x)$ satisfies 
 \begin{align}
 \calF[h](x) & \subset \bigtimes_{i=1}^n\calF[h_i](x) \\
 & \subset \bigtimes_{i=1}^n \sum_{j=1}^n  a_{ij}\calF[g_{ij}](\varphi_{ji}(x)).
 \end{align}
 By property 1 in Theorem 1 in \cite{paden1987}, for each $g_{ij}$ there exists a set $N_{g_{ij}}\subset \R^n$ with $\mu(N_{g_{ij}})=0$ such that
 \begin{equation}\label{e:equi_Fili_set_gij}
 \begin{aligned}
 \calF[g_{ij}](z) = \mathrm{co} \{\lim_{k \rightarrow \infty} h(z^k)
 \mid & \lim_{k \rightarrow \infty} z^k = z \textnormal{ and } \\
 & z^k\notin N_{g_{ij}}\cup N' \} 
 \end{aligned}
 \end{equation}
 for any set $N'$ with $\mu(N') = 0$.
 Similarly, there exists a set $N_h\subset \R^n$ with $\cup_{e_{ij}\in\calE} N_{g_{ij}}\subset N_h$ and $\mu(N_h)=0$ such that
 \begin{equation}\label{e:equi_Fili_set}
 \begin{aligned}
 \calF[h](x(t))=\mathrm{co} \{\lim_{k \rightarrow \infty} h(y^k) \mid & \lim_{k \rightarrow \infty}y^k= x(t),\\
 & y^k\notin N_h\cup S \},
 \end{aligned}
 \end{equation}  
where $S=\{x\in\R^n \mid \exists i,j\in\calI \textnormal{ such that } x_i=x_j \}$, which has measure zero in $\R^n$. Notice that $\R^n\setminus S$ admits a partition $\R^n \setminus S = S_1\cup S_2 \cup \cdots \cup S_{2^n}$, with $S_1,S_2, \ldots, S_{2^n}$  open sets satisfying $S_i\cap S_j=\emptyset$ for all $i\neq j$, such that within a fixed open set $S_i$, the components $y_1,y_2,\ldots,y_n$ of each vector $y\in S_i$ are all different and have the same fixed order.

Now, to study the right-hand side of \eqref{e:equi_Fili_set}, let $t$ be a given time and let $(y^k)$ be a sequence in $\R^n \setminus (N_h\cup S)$ that converges to $x(t)$ for which the limit $\tilde{h} := \lim_{k \rightarrow \infty} h(y^k)$ exists. Note that the existence of $\lim_{k\rightarrow \infty} h(y^k)$ means that all the components $\tilde{h}_i := h_i(y^k)$ have a limit.
We will study the term $\sum_{i\in\alpha(t)}\tilde{h}_i$ in order to derive that $\sum_{i\in\alpha(t)}\nu_i \leq 0$ for each $\nu\in\calF[h](x(t))$.
For this, we first define two sets of edges, namely 
\begin{align}
\calE_1(t) & = \{ e_{ij}\in\calE \mid i,j\in\alpha(t) \}, \\
\calE_2(t) & = \{ e_{ij}\in\calE \mid i\in\alpha(t), j\notin\alpha(t)\}.
\end{align}

The sequence $(y^k )$ has a subsequence $(y^{k_\ell})$ such that $y^{k_\ell}\in S_r$ for all $\ell$ for a fixed $r \in \{1,2,\ldots,2^n\}$. 
For an edge $e_{ij}\in\calE_1$, we have $y^{k_\ell}_i - y^{k_\ell}_j \uparrow 0$ or $y^{k_\ell}_i - y^{k_\ell}_j \downarrow 0$, depending on the set $S_r$.
Therefore, we have
\begin{equation*}
\begin{aligned}
&\lim_{\ell\rightarrow\infty}
[g_{ij}(\varphi_{ji}(y^{k_\ell})), g_{ji}(\varphi_{ij}(y^{k_\ell})) ]
= [g_{ij}(0^-), g_{ji}(0^+)]  \textnormal{ or } \\
&\lim_{\ell\rightarrow\infty} 
[g_{ij}(\varphi_{ji}(y^{k_\ell})), g_{ji}(\varphi_{ij}(y^{k_\ell}))]
= [g_{ij}(0^+), g_{ji}(0^-)].
\end{aligned}
\end{equation*}
Using Assumption \ref{as:symmetry_origin}, we see that in both cases we have
\begin{equation}\label{e:g_ijstuff}
 \lim_{\ell\rightarrow\infty} g_{ij}(\varphi_{ji}(y^{k_\ell})) +  g_{ji}(\varphi_{ij}(y^{k_\ell})) =0. \\
\end{equation}
Now, we can write
\begin{align} 
&\sum_{i\in\alpha(t)} \lim_{k\rightarrow\infty} h_i(y^k) \nonumber\\
 = & \lim_{\ell\rightarrow\infty} \sum_{i\in\alpha(t)} \sum_{j=1}^n a_{ij}g_{ij} (\varphi_{ji}(y^{k_\ell})) \nonumber\\
 = & \lim_{\ell\rightarrow\infty} \big[ \sum_{e_{ij}\in\calE_1} a_{ij}g_{ij} (\varphi_{ji}(y^{k_\ell})) + \sum_{e_{ij}\in\calE_2} a_{ij}g_{ij} (\varphi_{ji}(y^{k_\ell})) \big]\nonumber\\
 = & \frac{1}{2}\sum_{e_{ij}\in\calE_1}\lim_{\ell\rightarrow\infty}  a_{ij}g_{ij} (\varphi_{ji}(y^{k_\ell}))+a_{ji}g_{ji} (\varphi_{ij}(y^{k_\ell})) \nonumber\\
 &+ \lim_{\ell\rightarrow\infty} \sum_{e_{ij}\in\calE_2} a_{ij}g_{ij} (\varphi_{ji}(y^{k_\ell}))  \nonumber\\
 = & \lim_{\ell\rightarrow\infty} \sum_{e_{ij}\in\calE_2} a_{ij}g_{ij} (\varphi_{ji}(y^{k_\ell})),  \label{e:longestequation}
\end{align}
where the last two equalities are implied by the fact that the graph $\calG$ is undirected and by equation \eqref{e:g_ijstuff}. 
Next, we consider two possible cases: $x(t) \notin \spa\{\ones \}$, and $x(t)\in\spa\{\ones \}$. 

First, we look at the case that $x(t) \notin \spa\{\ones \}$, in which case $\calE_2 \neq \emptyset$. 
For an edge $e_{ij}\in\calE_2$ we have $x_j<x_i$, and since $g_{ij}$ is a sign-preserving function, this implies that  $\calF[g_{ij}](x_j-x_i) \subset \R_-$. 
As $y^{k_\ell}\notin N_{g_{ij}}$, all the accumulation points of the sequence $\{g_{ij} (\varphi_{ji}(y^{k_\ell})) \}$ belong to $\calF[g_{ij}](x_j-x_i)$. 
Therefore, we have that $\sum_{i\in\alpha(t)} \lim_{k\rightarrow\infty} h_i(y^k) = \lim_{\ell\rightarrow\infty} \sum_{e_{ij}\in\calE_2} a_{ij}g_{ij} (\varphi_{ji}(y^{k_\ell}))<0$, i.e., $\sum_{i\in\alpha(t)}\tilde{h}_i<0.$
By equation \eqref{e:equi_Fili_set}, we can conclude that $\sum_{i\in\alpha(t)}\nu_i<0$ for any $\nu\in\calF[h](x(t))$. Hence, by observation \eqref{e:nu-alpha}, we have $\tilde{\mathcal{L}}_{\calF[h]}V(x(t))\subset\R_-$. By the fact that $\tilde{\mathcal{L}}_{\calF[h]}V(x(t))$ is closed (see e.g. page 63 in \cite{cortes2008}), we have $\max \tilde{\mathcal{L}}_{\calF[h]}V(x(t))<0.$

Second, we consider the case that $x(t)\in\spa\{\ones\}$, in which case $\calE_1(t)=\calE$ and $\calE_2(t) = \emptyset$. In this case, equation \eqref{e:longestequation} boils down to
\begin{equation}
\sum_{i\in\alpha(t)}\tilde{h}_i =\sum_{i\in\calI}\lim_{k\rightarrow\infty} h_i(y^k) =0.
\end{equation} 
By equation \eqref{e:equi_Fili_set}, we can conclude that $\sum_{i\in\calI}\nu_i=0$ for any $\nu\in\calF[h](x(t))$. This implies that $\tilde{\mathcal{L}}_{\calF[h]}V(x(t))=\{0\}$ since $\frac{1}{n} \ones \in \partial V(x(t))$.

By using the same arguments as above, we can prove that
\begin{enumerate}[(i)]
\item $\max \tilde{\mathcal{L}}_{\calF[h]}W(x(t))<0$ if $x(t) \notin \spa\{\ones \}$,
\item $\tilde{\mathcal{L}}_{\calF[h]}W(x(t))=\{0 \}$ if $x(t) \in \spa\{\ones \}$. 
\end{enumerate}  

The above analysis implies that all trajectories are bounded. Indeed for any 
$N\in\R_+$ the set $S_N=\{x\in\R^n  \mid \|x\|_{\infty}\leqslant 
N \}$ is strongly invariant. By Theorem 2 in \cite{cortes2008}, the conclusion follows. 
\end{proof}

\begin{rem}
A stronger assumption is to assume that  $g_{ij}(y)=-g_{ji}(-y)$ for all $e_{ij} \in\calE$ and all $y\in \R$. This would imply that $\sum_{i\in\calI}h_i(x)=0$ for any $x\in\R^n$. Hence for any $x\in\R^n$ and for any $\nu\in\calF[h](x)$, we have $\ones^\top\nu=0$. Then any Filippov solution $x(t)$ of system \eqref{e:edge_nonlinear} satisfies $\ones^\top\dot{x}(t)=0$. Under the same assumption as in Theorem \ref{t:edge_nonlinear}, the trajectories of \eqref{e:edge_nonlinear} converge to a consensus value defined by the average of the initial condition.
\end{rem}

For the rest of this section, we consider \emph{directed} graphs. 
In this case, Assumption \ref{as:symmetry_origin} is not sufficient to guarantee convergence to 
consensus as shown by the following example.
 
\begin{exmp}\label{ex:counterexmaple_edgenonlinear_directed}
Consider system \eqref{e:edge_nonlinear} on the directed graph as in Fig.~\ref{fig:3nodesB}, where the functions $g_{ij}$ are the signum function. 
Suppose that at time $t_0$, the state satisfies $x(t_0)\in\spa\{\ones\}$. 
Then the Filippov set-valued map $\calF[h](x(t_0))$ is the same as in 
\eqref{e:filippovset-example}. Hence, by 
the same argument as in Example \ref{example_undi_3nodes}, there are Filippov solutions that exhibit sliding consensus.
\end{exmp}

For digraphs, we quote the following result from \cite{Papachristodoulou2010} for the case that the functions $g_{ij}$ are continuous.

\begin{thm}\label{Papachristodoulou2010}
 Consider the system \eqref{e:edge_nonlinear} with continuous functions $g_{ij}$. 
If the underlying graph $\calG=\{\calV, \calE \}$ contains a directed spanning tree, 
then the trajectories of \eqref{e:edge_nonlinear} achieve consensus asymptotically.
\end{thm}

Extension of Theorem \ref{Papachristodoulou2010} to the case of discontinuous functions $g_{ij}$ is a topic for further research.

\subsection{Combining node and edge nonlinearities} \label{ss:combine}
The multi-agent system given in \eqref{e:generalsystem} can be seen as a combination of system 
\eqref{e:node_nonlinear} and system \eqref{e:edge_nonlinear}.
For this system, we have the following result.

\begin{thm}\label{t:combined}
 Consider system \eqref{e:generalsystem} defined on a digraph 
$\calG=\{\calV,\calE \}$, with continuous functions $g_{ij}$. If one of the 
following three conditions holds, i.e., 
\begin{enumerate}[(i)]
\item $\calI_r\cap\calI_c$ is not empty,
\item $|\calI_r|=1$,
\item $|\calI_r|=2$, $f_i(0^-)$ and $f_i(0^+)$ exist, and $f_i(0^-)=-f_i(0^+)$ for $i\in\calI_r$, 
\end{enumerate}
then all Filippov solutions of system \eqref{e:node_nonlinear}
 achieve consensus asymptotically, for all initial conditions.
\end{thm}

\begin{proof}
Since the proof is similar to the proof of Theorem \ref{t:main}, 
we only provide a sketch of the proof. Recall that $\alpha(t)=\{i\in\calI\mid x_i(t)=V(x(t))\}$  and $\beta(t)=\{i\in\calI\mid x_i(t)=-W(x(t))\}$.
Let $V$ and $W$ be candidate Lyapunov functions. We will show that 
$\max \tilde{\mathcal{L}}_{\calF[h]}V\leqslant 0$ by considering two 
cases: $\calI_r\nsubseteq\alpha(t)$ and $\calI_r\subseteq\alpha(t)$. 

When $\calI_r\nsubseteq\alpha(t)$, there exists at least one $k\in\alpha(t)$ 
with $\sum_{j=1}^n a_{kj}g_{kj}(x_j-x_k)<0$, which implies the $k$th 
component of $\calF[h](x(t))$ is in $\R_-$. 
Hence, $\max \tilde{\mathcal{L}}_{\calF[h]}V<0$.

When $\calI_r\subseteq\alpha(t)$,  we can use similar arguments as in the proof of Theorem \ref{t:main} to see that
the set-valued Lie derivative $\tilde{\mathcal{L}}_{\calF[h]}V(x(t))$ is either 
$\{0\}$ or $\emptyset$ if one of the conditions {\it (i)}, {\it 
(ii)} and {\it (iii)} holds. Hence $\max 
\tilde{\mathcal{L}}_{\calF[h]}V(x(t))\leqslant0$.

Similarly, we have that $\max \tilde{\mathcal{L}}_{\calF[h]}W(x(t))<0$ if $\calI_r\nsubseteq\beta(t)$,  
and  $\max \tilde{\mathcal{L}}_{\calF[h]}W(x(t))\leqslant 0$ if $\calI_r\subseteq\beta(t)$.
Based on Theorem 2 in \cite{cortes2008}, the conclusion follows.
\end{proof}

\section{A port-Hamiltonian perspective on consensus error dynamics}\label{s:PHformulation}

An alternative approach to consensus analysis is to consider the dynamics of the `error' vector $z=-Lx$. In many cases, the convergence of $x$ to ordinary, \emph{static}, consensus is equivalent to the convergence of $z$ to the origin. On the other hand, in the previous section it was shown that for differential inclusions this equivalence does not necessarily hold. In the present section we first provide sufficient conditions which guarantee asymptotic stability of the origin for the error dynamics for compatible initial conditions. Combining this with Theorem \ref{t:main}, we then formulate sufficient conditions for the equivalence between convergence of $z$ to the origin and of $x$ to static consensus.

\begin{thm}\label{t:pH}
Given system \eqref{e:node_nonlinear_fili},  the error $z = - Lx$ satisfies
\begin{align}\label{e:error}
\dot{z}\in -L\calF[f](z).
\end{align} 
If $\calG$ is strongly connected and the functions $f_i$ are sign-preserving and non-decreasing, then all Filippov solutions of \eqref{e:error} for $z(0)\in\im L$ converge to the origin.
\end{thm}

\begin{proof}
By almost everywhere differentiability of $z(\cdot)$
\begin{align*}
\dot{z}(t) & = -L\dot{x}(t) \\
& \in  -L \calF[f(-Lx)](x(t)) \\
& \subset -L \calF[f](z(t)),
\end{align*}
where the last inclusion holds by \cite[Theorem 1 (1)]{paden1987} for any $z(t)=-Lx(t)$. 
Next, we prove asymptotic stability of the origin for any $z(0)\in\im L$. We only provide the sketch of the proof. First denote $F(y)=[F_1(y_1),F_2(y_2),\ldots, F_n(y_n)]^\top$ for any $y\in\R^n$, where $F_i(y_i):=\int_{0}^{y_i}f_i(s)ds$ is convex and radially unbounded, since $f_i$ is sign-preserving and non-decreasing. Furthermore, since $\calG$ is strongly connected, there exists $\sigma\in\R_+^n$ such that $\sigma^\top L=0$ \cite[Theorem 14]{Bollobas98}. Consider $V_1(z)=\sigma^\top F(z)$ as Lyapunov function, which is convex and hence regular. Then since $F_i$ is differentiable almost everywhere Theorem 1 (1) in \cite{paden1987} implies that the generalized gradient of $V_1$ is 
\begin{align}
\partial V_1(z) & = \bigtimes_{i=1}^n [\sigma_if_i(z^-_i),\sigma_if_i(z^+_i)] \\
& =  \Sigma \calF[f](z),
\end{align}
where $\Sigma = \diag (\sigma_1,\ldots, \sigma_n)$.
Then for any $a\in \tilde{\mathcal{L}}_{-L\calF[f]}V_1(z)$ there exists $\nu\in\calF[f](z)$ such that $a = -\nu^\top \Sigma L\nu $.  Since $\Sigma L$ is the Laplacian matrix of a balanced graph, and the symmetric part of $\Sigma L$ is positive semidefinite \cite[Theorem 1.37]{DistCtrlRobotNetw}, we have $a \leq 0$. Thus we have shown that if $\tilde{\mathcal{L}}_{-L\calF[f]}V_1(z)\neq \emptyset$, then it belongs to $\R_{\leq 0}$.
Furthermore, by Theorem 2 in \cite{cortes2008} all solutions of \eqref{e:error} converge to $\Omega =\overline{\{z \mid 0\in\tilde{\mathcal{L}}_{-L \calF[f]}V_1(z)\}}$ asymptotically. By computing  $\tilde{\mathcal{L}}_{-L \calF[f]}V_1(z)$ it follows that $\Omega=\overline{\{z \mid  \spa 
\{\ones\}\subset \calF[f](z) \}}$. Since $z(t) \in\im L$ for all $t$ and the functions $f_i$ are sign-preserving, we obtain $\Omega=\{0\}$.

\end{proof}

\begin{rem}
The stability of the system \eqref{e:error} can be approached from the following point of view. By using the new coordinates  $w=\Sigma z$ we can write 
\begin{align}
 \dot{w}  \in  -\Sigma L \frac{\partial H}{\partial w}(w) = -(J+R)\frac{\partial H}{\partial w}(w), \label{e:w-PH}
\end{align}
where $H(w):=\sigma^\top F(\Sigma^{-1}w)$, and $J$ and $R$ are the skew-symmetric and symmetric parts of $\Sigma L$. The system \eqref{e:w-PH} is a generalized (differential inclusion) port-Hamiltonian system \cite{Arjan2014ph}. Thus the Lyapunov function $V_1$ in the new coordinates $w$ is nothing else than the Hamiltonian of this port-Hamiltonian system. 
\end{rem}

\begin{exmp} 
The system \eqref{e:node_nonlinear} and the error dynamics \eqref{e:error} resulting from $z=-Lx$ can be illustrated as follows. 
Consider a hydraulic network with $x_i$ being the pressure at the $i$-th node, where the flow through the pipe from node $i$ to node $j$ is linearly dependent on $x_i-x_j$. Then the flow extracted/injected at the $i$-th node equals the $i$-th component of $z=-Lx$, where $L$ is a symmetric Laplacian matrix. System \eqref{e:node_nonlinear} is obtained by assuming that the rate of increase of $x_i$ depends on $z_i$ through the function $f_i$. Obviously the pressures will converge to consensus if and only if the vector of flows $z$ at the nodes converges to zero. $H$ is the storage function resulting from summing the integrals of $f_i$. Non-symmetric Laplacian matrices $L$ may occur in other types of transportation networks \cite{SchSC15}.
\end{exmp}

\begin{exmp}\label{example_undi_3nodes_continue}
Consider the system in Example \ref{example_undi_3nodes}. Its error dynamics is given as 
$\dot{z}\in -L\calF[\sign](z).$
Hence $\lim_{t\rightarrow\infty}z(t)=0, \forall z(0)\in\im L$. Indeed, following the proof of Theorem \ref{t:pH} the Lyapunov function is the $1$-norm $V_1(z)=\|z\|_1$.   
\end{exmp}

As shown by Examples \ref{example_undi_3nodes} and \ref{example_undi_3nodes_continue} the convergence of the error vector $z$ to zero does not necessarily imply the convergence of $x$ to consensus, since `sliding' consensus may occur. 
Hence, one must be careful to derive the convergence of $x$ to consensus by analyzing the error vector $z$.  In the following remark, we combine Theorem \ref{t:main} and \ref{t:pH} to obtain sufficient conditions for when convergence of the error vector $z$ to the origin guarantees (static) consensus of $x$.

\begin{rem}
Consider systems \eqref{e:node_nonlinear_fili} and \eqref{e:error} defined on a strongly connected digraph. Suppose the functions $f_i$ are sign-preserving and non-decreasing. Then if one of the following conditions hold 
\begin{enumerate}[(i)]
\item $\calI_c$ is non-empty,
\item 
$|\calI|=2$, $f_i(0^-)$ and $f_i(0^+)$ exist and $f_i(0^-)=-f_i(0^+)$ for $i\in\calI$, 
\end{enumerate}
then $z(t)$ converges to the origin for any $z(0)=-Lx(0)$ and $x(t)$ converges to consensus for any $x(0)$.
\end{rem}

\section{Conclusion}\label{s:conclusion}

In this paper, we considered a very general model of multi-agent  systems defined on a directed graph, with nonlinear discontinuous functions defined on the nodes and edges.  
Since the right-hand sides of the differential equations are discontinuous, we interpreted the solutions in the Filippov sense. Under the crucial assumptions of $(i)$ the graph containing a directed spanning tree, $(ii)$ all nonlinear functions to be sign-preserving, we provided sufficient conditions for all Filippov solutions of the consensus protocol to achieve consensus asymptotically. 
Furthermore, a common approach in the study of consensus, namely the error dynamics, was analyzed. 
A sufficient condition was given to guarantee the convergence of the error to zero. For differential inclusions, however, it was shown that convergence of the error to zero is not equivalent to the convergence of the original states to consensus, and that one has to be careful in order to avoid `sliding' consensus.

\bibliographystyle{plain}        
\bibliography{ifacconf}

\end{document}